\documentclass[a4paper,reqno,12pt]{amsart}
\usepackage[utf8]{inputenc}
\usepackage[T1]{fontenc}
\usepackage[english]{babel}
\usepackage{amsfonts,amsmath,amssymb,amsthm}

\usepackage[expansion=false]{microtype}
\usepackage{mathtools}
\usepackage{geometry}
\usepackage{textcomp}
\usepackage[usenames]{color}
\usepackage{graphicx}
\usepackage{faktor}
\usepackage{xfrac}
\usepackage{euscript}
\usepackage{tikz}
\usepackage{tikz-cd}
\usepackage{enumitem}
\usepackage[final]{hyperref}
\hypersetup{
  colorlinks=true,
  linkcolor=black,
  citecolor=black,
  urlcolor=black
}

\theoremstyle{plain}
\newtheorem{theorem}{Theorem}[section]

\newtheorem{corollary}[theorem]{Corollary}

\newtheorem{proposition}[theorem]{Proposition}
\theoremstyle{definition}
\newtheorem{definition}[theorem]{Definition}
\newtheorem{notations}[theorem]{Notation}

\newtheorem{construction}[theorem]{Construction}

\theoremstyle{plain}
\newtheorem{lemma}[theorem]{Lemma}

\theoremstyle{remark}
\newtheorem{remark}[theorem]{Remark}

\DeclareMathOperator{\Spec}{Spec}

\DeclareMathOperator{\ML}{ML}

\DeclareMathOperator{\Hom}{Hom}

\DeclareMathOperator{\Quot}{Quot}

\newcommand{\KK}{\mathbb K}
\newcommand{\A}{\mathbb A}
\newcommand{\TT}{\mathbb{T}}

\newcommand{\ZZ}{\mathbb Z}
\newcommand{\QQ}{\mathbb Q}

\DeclareMathOperator{\cone}{cone}

\newcommand{\lmt}{\longmapsto}
\newcommand{\DD}{\mathfrak D}
\newcommand{\OO}{\mathcal O}
\newcommand{\PP}{\mathbb P}
\newcommand{\kk}{\Bbbk}

\DeclareMathOperator{\divisor}{div}
\newcommand{\HH}{\mathbb H}

\newcommand{\id}{\mathrm{id}}
\newcommand{\ii}{\mathbf{i}}

\usepackage{mathtools}

\makeatletter
\DeclareRobustCommand\widecheck[1]{{\mathpalette\@widecheck{#1}}}
\def\@widecheck#1#2{%
    \setbox\z@\hbox{\m@th$#1#2$}%
    \setbox\tw@\hbox{\m@th$#1%
       \widehat{%
          \vrule\@width\z@\@height\ht\z@
          \vrule\@height\z@\@width\wd\z@}$}%
    \dp\tw@-\ht\z@
    \@tempdima\ht\z@ \advance\@tempdima2\ht\tw@ \divide\@tempdima\thr@@
    \setbox\tw@\hbox{%
       \raise\@tempdima\hbox{\scalebox{1}[-1]{\lower\@tempdima\box
\tw@}}}%
    {\ooalign{\box\tw@ \cr \box\z@}}}
\makeatother

\title{$\mathbb T$-homogeneous locally nilpotent derivations of trinomial algebras}
\date{}
\author{Timofey Krasikov and Kirill Rassolov}
\address{%
    \begin{minipage}{\textwidth}
        \setlength{\parindent}{0pt}
        Timofey Krasikov\\
        \textbf{E-mail:} \url{timkrasikov@gmail.com}\\
        Kirill Rassolov\\
        \textbf{E-mail:} \url{kirill.rassolov@math.msu.ru}\\[4pt]
        Lomonosov Moscow State University, Faculty of Mechanics and Mathematics,\\
        Department of Higher Algebra, Leninskie Gory 1, Moscow, 119991 Russia;\\[2pt]
        and\\[2pt]
        HSE University, Faculty of Computer Science,\\
        Pokrovsky Boulevard 11, Moscow, 109028 Russia
    \end{minipage}%
}


\begin{document}

\begin{abstract}
    A trinomial algebra is a commutative finitely generated algebra given by a system of compatible relations each of which is a polynomial with three terms. Such algebras arise as the Cox rings of varieties admitting a complexity one torus action. We describe locally nilpotent derivations of a trinomial algebra that are homogeneous under a natural torus action of complexity one.
\end{abstract}

\thanks{The work is supported by the grant RSF 25-11-00302}

\maketitle

\section*{Introduction}
A \textit{trinomial variety} is an affine variety given by a system of equations of the form $aM_i+bM_{i+1}+cM_{i+2}=0$, compatible in a special way. Here the $M_i$ denote fixed monomials in disjoint sets of variables; see Section~\ref{SubsectionTrinomials} for details. Following~\cite{HW}, we distinguish two types of trinomial varieties depending on whether the monomial corresponding to the empty set of variables (i.e., equal to $1$) is allowed. The coordinate algebra of a trinomial variety is called a \textit{trinomial algebra}. The ground field $\kk$ is assumed algebraically closed and of characteristic zero.

Recall that the \emph{complexity} of an action of a reductive group $G$ on a variety $X$ is the codimension of a general orbit of a Borel subgroup $B\subseteq G$. For an effective action of an algebraic torus $T$, the complexity equals $\dim X - \dim T$. A trinomial variety is endowed with a natural torus action of complexity one, which comes from the standard torus action on affine space by scaling the variables. Moreover, there is a close connection between trinomial algebras and arbitrary varieties with a torus action of complexity one via Cox rigns. 

The Cox construction (see, e.g.,~\cite{ADHL}) is fruitfully applied to toric varieties, that is to say, normal varieties with a torus action of complexity zero. The Cox ring of a toric variety is a polynomial ring, which allows one to realize, say, an affine toric variety as the categorical quotient of an affine space by an action of a diagonalizable group~\cite{C}. In turn, trinomial algebras are the Cox rings of normal rational semiprojective varieties with a torus action of complexity one, finitely generated divisor class group, and with no nonconstant invertible functions. In particular, every affine variety satisfying these conditions is the categorical quotient of a trinomial variety by an action of a diagonalizable group~\cite{HH, HW}.

Consider a finitely generated commutative associative algebra $R$ over the field $\kk$. A \emph{derivation} is a linear map $\delta\colon R\to R$ satisfying the Leibniz rule: $\delta(fg) = \delta(f)g+f\delta(g)$ for all $f,g\in R$.
Suppose $R$ is equipped with a grading $R=\bigoplus_{g\in G} R_g$ by an abelian group $G$. A derivation $\delta$ is called {\it homogeneous} if there exists an element $d\in G$, called the \emph{degree} of~$\delta$, such that $\delta(R_g)\subseteq R_{g+d}$ for all $g\in G$. It is well known that specifying a grading on $R$ by a finitely generated abelian group is equivalent to specifying an action of a quasitorus, i.e., a diagonalizable group. Moreover, gradings by free abelian groups correspond to torus actions. We say that a derivation~$\delta$ is \emph{homogeneous} under a quasitorus if $\delta$ is homogeneous with respect to the corresponding grading.

A derivation $\delta$ is called \emph{locally nilpotent}, or an \emph{LND}, if for every $f\in R$, there exists an $n\in \ZZ_{>0}$ such that $\delta^n(f)=0$. The exponential map establishes a bijection between locally nilpotent derivations up to scaling and $\mathbb G_a$-subgroups in the automorphism group of a given algebra. By a $\mathbb G_a$-subgroup we mean a one-parameter algebraic subgroup isomorphic to the additive group of the ground field $\kk$. Namely, an LND $\delta$ corresponds to the subgroup $\{\exp(t\delta)\,|\,t\in \kk\}$. Furthermore, LNDs homogeneous with respect to the action of a quasitorus $H$ correspond to $H$-normalized $\mathbb G_a$-subgroups.

Homogeneous LNDs of trinomial algebras were studied in~\cite{Z, GZ, R}. Explicit formulae for the homogeneous LNDs were obtained in the case of complexity one action of a special quasitorus $\HH$ coming from the standard torus action on affine space by scaling the variables. The $\HH$-homogeneous LNDs were called \emph{elementary}. In the present paper, we consider LNDs homogeneous with respect to the torus ${\TT}$ which is the neutral component of the group $\HH$. For factorial varieties, the quasitorus $\HH$ is connected and therefore coincides with~${\TT}$. However in non-factorial case we have the strict inclusion ${\TT}\subset\HH$, and there may appear ${\TT}$-homogeneous LNDs which are not $\HH$-homogeneus; see Remark~\ref{RemarkCompareT&H} for specific example. 

Our interest in LNDs homogeneous with respect to a torus is motivated by their applications to study of arbitrary LNDs. Given a derivation of a finitely generated algebra, one can decompose it into a sum of homogeneous derivations, which are called the {\it homogeneous components}. However, the homogeneous components of a locally nilpotent derivation are, generally speaking, not locally nilpotent. The importance of free gradings is that for certain homogeneous components of a given LND, local nilpotency can be guaranteed. To be more precise, one should consider the Newton polytop of the degrees of the non-zero homogeneous components; then the homogeneous components corresponding to the vertices turn out to be locally nilpotent. Thus the study of homogeneous LNDs for free gradings provides more information about an arbitrary LND than for gradings by a non-free abelian group. 
In particular, we give a new proof of the rigidity criterion obtained earlier in~\cite{G}.

Recall that an algebra $R$ is called \textit{rigid} if its automorphism group contains no $\mathbb G_a$-subgroups, i.e., $R$ admits no non-zero LNDs. Two LNDs are called \textit{equivalent} if their kernels, i.e., the algebras of invariants of the corresponding $\mathbb G_a$-subgroups, coincide. $R$ is called \emph{semirigid} if all nonzero LNDs are equivalent.

We show that for a trinomial algebra of Type~$1$, all ${\TT}$-homogeneous LNDs are replicas of elementary LNDs, and we provide explicit formulae (Theorem~\ref{TheoremType1}). 
As an application, besides a new proof of the rigidity criterion (Corollary~\ref{CorRigidiryType1}), we also obtain a semirigidity criterion for Type~$1$, see Corollary~\ref{CorSemiRidigType1}.

For a trinomial algebra of Type $2$, we obtain a description of the equivalence classes of ${\TT}$-homogeneous LNDs (Theorem~\ref{TheoremType2Classes}). This implies a semirigidity criterion, see Corollary~\ref{CorSemiRigidType2}. Finally, Theorem~\ref{TheoremType2} provides explicit formulae for ${\TT}$-homogeneous LNDs in the case when the algebra is defined by a single relation. In contrast to Type~$1$, there appear ${\TT}$-homogeneous LNDs that are not replicas of elementary ones.

Our proofs for Type~$2$ use the combinatorial description of homogeneous LNDs on an affine variety with a torus action of complexity one from~\cite{Liendo}. It is based on the representation of varieties with a torus action via proper polyhedral divisors on semiprojective varieties, see~\cite{Altmann-Hausen}. However when a high-dimensional affine variety is explicitly given by equations, this requires complex computations associated with multidimensional rational polyhedra. Therefore, we use the technique of polyhedral divisors only for trinomial surfaces. Next, we show that a ${\TT}$-homogeneous LND on an arbitrary trinomial variety of Type~$2$ corresponds to a homogeneous LND on a trinomial surface.

\section{Preliminaries}
\subsection*{Trinomial algebras}\label{SubsectionTrinomials}
We begin with some notation. Fix integers $r\ge2,$ $d\ge0, r_0\in\{0,1\}$ and consider the variables $T_{ij},\,S_k$ with $i\in\{r_0,\ldots,r\},$ $j\in\{1,\ldots,n_i\},$ $n_i\ge1$, $k\in\{1,\ldots, d\}$. We will view the $T_{ij}, S_k$ as coordinates in the affine space $\A^{n+d}$, $n = n_1+\ldots+n_r$, and abbreviate the corresponding polynomial algebra in $n+d$ variables by $\kk[T_{ij}, S_k]$. Following Construction 1 of~\cite{HW}, we now define two types of trinomial varieties.

\textit{Type $1$.} Set $r_0=1$ and choose pairwise distinct elements $a_i\in\kk$. Then the subvariety in the affine space $\mathbb A^{n+d}$ given by:
\begin{equation}\label{equations1}
    T_i^{l_i}-T_{i+1}^{l_{i+1}} - (a_{i+1}-a_i)=0,\quad i=1,\ldots, r-1
\tag{1.1}
\end{equation}
is called a \emph{trinomial variety of Type $1$}. Note that~\eqref{equations1} is equivalent to that any three columns of the matrix:
$$
M = \begin{pmatrix}
    1 & T_{1}^{l_1}&\ldots&T_r^{l_r}\\
    -1 & a_1 &\ldots & a_r\\
    0 &1&\ldots&1
\end{pmatrix}
$$
are linearly dependent over $\Bbbk$.

\textit{Type $2$}. Set $r_0=0$. Fix also a matrix $A = (a_0,\ldots,a_r)$ with pairwise linearly independent column vectors $a_i\in\kk^2$. We define a \emph{trinomial variety of Type $2$} 
in $\mathbb A^{n+d}$ 
by requiring that any three columns of the $3\times(r+1)$-matrix:
$$
M=\begin{pmatrix}
    T_0^{l_0} &\ldots &T_r^{l_r}\\
    a_0 &\ldots &a_r
\end{pmatrix}
$$
are linearly dependent over $\Bbbk$. In turn, this can be rewritten as
\begin{equation}\label{equations2}
    \det \begin{pmatrix}
    T_i^{l_i} & T_{i+1}^{l_{i+1}} & T_{i+2}^{l_{i+2}}\\
    a_i & a_{i+1} & a_{i+2}
\end{pmatrix}=0, \quad i=0,\ldots,r-2.
\tag{1.2}
\end{equation}
 \begin{definition}\label{DefTrinomial}
     \emph{A trinomial algebra} of Type $1$ (resp.~$2$) is the quotient algebra of $\kk[T_{ij}, S_k]$ by the ideal generated by relations~(\ref{equations1}) (resp.~(\ref{equations2})).
 \end{definition}

 It was proved in~\cite[Theorem~1.2]{HW} that any trinomial algebra is a normal domain of dimension $n+d-r+1$. Thus, it is the algebra of global functions of the corresponding trinomial variety~$X$; the latter is also normal and of dimension $n+d-r+1$. We endow $X$ with a complexity one torus action as follows.

\begin{construction}\label{ConstrAct}
Consider the torus of dimension $n+d$ acting on the variables $T_{ij}$ and $S_k$ in $\A^{n+d}$ by multiplication. Let $\HH$ be the stabilizer of equations (\ref{equations1}) or~(\ref{equations2}) respectively; it is given by requiring that the monomials $T_i^{l_i}$ are invariant in case of Type $1$, and semi-invariant, of the same weight in case of Type $2$. Let ${\TT}$ stand for the identity component of $\HH$. Clearly, ${\TT}$ and $\HH$ act on the trinomial variety $X$ and $\dim {\TT}=\dim X-1$.
\end{construction}

\begin{remark}
The $\HH$-action corresponds to a grading of $\kk[X]$ by the character group of $\HH$; this grading was defined in~\cite[Construction~1.1]{HW}. It can be viewed as the finest grading such that the variables $T_{ij}, S_k$ are homogeneous; see~\cite[Lemma~1]{Z}. The ${\TT}$-action is obtained by identifying the character lattice of ${\TT}$ with the torsion-free part of the character group of $\HH$.
\end{remark}
\begin{remark}\label{RemarkAct}
One can obtain the grading corresponding to the ${\TT}$-action more directly. Namely, fix a variable $T_{ib_i}$ in each monomial $T_i^{l_i}$.

In the case of Type $1$, we consider the group $\ZZ^{n+d-r}$ with the canonical basis $e_{ij}, \,e_k$ where $i\in\{1,\ldots, r\}$, $j\in\{1,\ldots, n_i\}\setminus \{b_i\}$, $k\in\{1,\ldots, d\}$. Set
 $$
    \deg T_{ij}=e_{ij}\prod_{k\neq j} l_{ik}, \quad j\neq b_i, \qquad\qquad\deg T_{ib_i} = -\sum_{j\neq b_i}e_{ij}\prod_{k\neq b_i} l_{ik},$$    
    $$
    \deg S_k=e_k,\quad\mbox{all $k$}.
    $$

    For Type $2$, consider the group $\mathbb Z^{n+d-r}$ with the canonical basis $e_{ij}, \,e_k,\, e$ where $i\in\{0,\ldots, r\}$, $j\in\{1,\ldots, n_i\}\setminus \{b_i\}$, $k\in\{1,\ldots, d\}$. 
    Set
    $$
    \deg T_{ij}=e_{ij}\prod_{k\neq j} l_{ik}, \quad\mbox{$j\neq b_i$,}\quad\qquad\deg T_{ib_i} = e\prod_{p\neq i} l_{pb_p}-\sum_{j\neq b_i}e_{ij}\prod_{k\neq b_i} l_{ik}, $$$$
    \deg S_k=e_k,\quad\mbox{all $k$}.
    $$

    Then one can think of the character lattice of ${\TT}$ as the subgroup (of finite index) of $\mathbb Z^{n+d-r}$ generated by all $\deg T_{ij},$ $ \deg S_k$.
\end{remark}
\begin{remark} 
Suppose that $n_il_{ij}>1$ for all $i,j$; this always can be achieved by eliminating the variables that occur as a linear term in some relation. It follows from~\cite[Theorem 1.2\,$(iv)$]{HW} that $\HH=\mathbb T$ (i.e., $\HH$ is connected) if and only if the trinomial algebra is factorial. Furthermore, in case of Type $1$, this is equivalent to the condition $\gcd(l_{i1},\ldots,l_{in_i})=1$ for all $i$. In case of Type~$2$, a necessary and sufficient condition is that the integers $d_i=\gcd(l_{i1},\ldots,l_{in_i})$ are pairwise coprime.
\end{remark}

We will say that nonzero nonunit elements $a,b$ of a trinomial algebra $R$ are \emph{coprime} if, for all $f,g\in R$, the condition $fa=gb$ implies that $a\mid g$ (hence also $b\mid f$).
\begin{lemma}\label{LemmaCoprime}\label{LemmaCoprimeFoolly}
    The variables $T_{ij}, S_k$ are pairwise coprime.   
\end{lemma}
\begin{proof}
    The grading by the character group of $\HH$ is factorial and the $T_{ij}, S_k$ are homogeneously prime; see~\cite[Theorem~1.2\,$(ii,\,iii)$]{HW}. Thus the assertion follows from the following fact: Let $R$ be a factorially graded domain and $a,b\in R$ nonassociated homogeneously primes; then $a,b$ are coprime. 
    
    In fact, take $f,g\in R$. Consider the homogeneous decomposition $f=f_1+\ldots+f_k$ and $g=g_1+\ldots+g_l$ where $f_p\neq0,\,g_q\neq0$ for all $p,\,q$. Then $fa=gb$ implies that $k=l$ and, after a suitable permutation of the $g_p$'s, we have $f_pa=g_{p}b$ for all $p$. Hence $a\mid g_{p}$ for all $p$, and so $a\mid g$.
\end{proof}
When an integer $a$ divides $l_{ij}$ for all $j$, we will write 
$T_i^{l_i/a}$ rather than $\prod_{j=1}^{n_i} T_{ij}^{l_{ij}/a}$. 

\begin{lemma}\label{LemmaRationalInvariants}
Let $R$ be a trinomial algebra of Type $2$. Then any element of the field of rational invariants $(\Quot R)^{{\TT}}$ is a rational function in
$$\frac{T_i^{l_i/d_{ij}}} {T_j^{l_j/d_{ij}}}\qquad \mbox{where}\quad d_{ij}=\gcd(d_i, d_j), \quad d_i=\gcd(l_{i1},\ldots, l_{in_i}).$$
\end{lemma}
\begin{proof}
    Take a rational ${\TT}$-invariant $f/g\in(\Quot R)^{\TT}$. We may assume both $f$ and $g$ to be ${\TT}$-semi-invariant; see~\cite[Theorem 3.3]{PV}. The assertion will follow from considering the weight of each monomial of $f$ and $g$.

    Note that, for any $i$, ${\TT}$ contains the $(n_i-1)$-dimensional subtorus ${\TT}_i$ acting only on the variables that appear in $T_i^{l_i}$. 
    Fix $i$ and consider a monomial $P$ in the variables $T_{ij}$. Then $P$ has weight zero (with respect to ${\TT}_i$) if and only if $P$ is a power of $T_i^{l_i/d_i}$. More generally, regardless of the weight of $P$, we have $P=Q\cdot(T_i^{l_i/d_i})^\alpha$, where $Q$ is a monomial not divisible by $T_i^{l_i/d_i}$, and $\alpha\in\mathbb Z_{\ge0}$.


    Since, for any $i$, $f$ and $g$ are of the same ${\TT}_i$-weight, these $Q$ are the same for any monomial appearing in $f$ and $g$ and thus may be cancelled. Repeating for all $i$, we represent both $f$ and $g$ as polynomials in the $T_i^{l_i/d_i}$ whose monomials are all of the same ${\TT}$-weight.

    Finally, consider the one-dimensional torus $\kk^\times\subset{\TT}$ acting only on the variables $T_{ib_i}$ (notation as in Remark~\ref{RemarkAct}). Two monomials
    $(T_i^{l_i/d_i})^\alpha$ and $(T_j^{l_j/d_j})^\beta$ are of the same $\kk^\times$-weight if and only if $\alpha/d_i=\beta/d_j$. Thus the ratio of these monomials is a power of 
   $T_i^{l_i/d_{ij}}/ T_j^{l_j/d_{ij}}$.
   
   It remains to divide both $f$ and $g$ by any one of the monomials appearing among them to rewrite $f/g$ as a rational function in the ${T_i^{l_i/d_{ij}}}/ {T_j^{l_j/d_{ij}}}$.
\end{proof}

\subsection*{Locally nilpotent derivations}
The following lemma gathers some basic properties of locally nilpotent derivations. We refer the reader to~\cite[Principles 1, 5, and 7 and Corollary 1.23]{Freu} for the proof.

\begin{lemma}\label{BasicLem}
    Let $R$ be a $\kk$-domain, $\delta$ an LND of $R$ and $f, g \in R.$ Then:
    \begin{enumerate} 
    \item[$(i)$] $f\delta$ is a derivation; it is locally nilpotent if and only if $f\in\ker\delta$.
    \item[$(ii)$] $\ker\delta$ is factorially closed, meaning that $fg \in \ker\delta$ implies $f, g\in\ker\delta$.
    \item[$(iii)$] Suppose that $f\, | \, \delta(f)$; then $\delta(f)=0$.
    \item[$(iv)$] Suppose that $f\, | \, \delta(g)$ and $g\, | \, \delta(f)$; then either $\delta(f)=0$ or $\delta(g)=0$.
    \end{enumerate} 
\end{lemma}
The locally nilpotent derivation $f\delta$, $f\in \ker\delta$, is called a \emph{replica} of $\delta$.

Recall that an Abelian group $G$ is called \emph{ordered} if it is endowed with a total order which is translation invariant, i.e., $a\le b$ implies $a+c\le b+c$ for any $a,b,c\in G$. Clearly, a finitely generated Abelian group can be ordered if and only if it is torsion-free. The following lemma is a particular case of~\cite[Lemma~1.10]{Liendo}.
\begin{lemma}\label{LemmaLNDHomDecomposition}
    Let $G\simeq\ZZ^k$ be an ordered group and $R$ a $G$-graded algebra. Then any LND $\delta\colon R\to R$ admits the decomposition
    $$
    \delta=\sum_{k=1}^p\delta_k, \qquad\deg\delta_1<\ldots<\deg\delta_p
    $$
     where $\delta_k\neq0$ are $G$-homogeneous derivations, and $\delta_1$, $\delta_p$ are LNDs.

    In particular, $R$ is rigid if and only if it admits no nonzero $\mathbb Z^k$-homogeneous LND.
\end{lemma}
For the end of this section, let us state some specific properties of ${\TT}$-homogeneous LNDs. 
\begin{lemma}\label{LemmaHLND}
    Let $\delta$ be a $\mathbb T$-homogeneous LND of a trinomial algebra. Then
    \begin{enumerate}
        \item [$(i)$] One has either $\delta(T_{ij})=0$ for all $i,j$, or $\delta(S_k)=0$ for all $k$;
        \item [$(ii)$] Suppose that $\delta(T_{ij})=0$ for all $i,j$; then there exists a $k$ such that $$\delta = \delta(S_{k}){\partial\over\partial S_{k}};$$
        \item[$(iii)$] For each monomial $T_i^{l_i}$ there exists a variable $T_{ic_i}$ such that $\delta(T_{ij})=0$ whenever $j\neq c_i$.
    \end{enumerate}
\end{lemma}
The lemma is proved in~\cite[Lemmas~6, 7, 8]{R} for the case when $\delta$ is an $\HH$-homogeneous LND. The proof proceeds by considering the gradings by certain subgroups of the character group of $\HH$. One readily verifies that these subgroups are in fact contained in the character lattice of ${\TT}$, so the same argument applies to the case of ${\TT}$-homogeneous LNDs.

\section{Main result on trinomial algebras of Type $1$}
\begin{construction}\label{ConstrDer1}
   Let $R$ be a trinomial algebra of Type~$1$. Suppose that there exists a tuple $C = (c_1, \ldots, c_r)$, $1\le c_i \le n_i$, such that $l_{ic_i}=1$ for all $i$ except maybe one $i_0$. We set
         \begin{equation*}
             \begin{split}
                 \delta_C(T_{ij}) &= \begin{cases}
                 \prod\limits_{k\neq i}\frac {\partial T_k^{l_k}}{\partial T_{kc_k}},&\text{$j=c_i$,}\\
                 0, &\text{$j \neq c_i$;}
                 \end{cases}\\
                 \delta_C(S_k) &= 0,\quad \text{all } k.
             \end{split}
         \end{equation*}
        There is a unique way to extend this assignment to a derivation of $\kk[T_{ij}, S_k]$ by the Leibniz rule. It is straightforward to show that the derivations obtained in this way are locally nilpotent (being nilpotent at the variables), well-defined on the quotient algebra $R$ (taking the relations between the $T_{ij}$ to $0$), and ${\TT}$-homogeneous (taking the variables to ${\TT}$-homogeneous elements). 
\end{construction}

\begin{theorem}\label{TheoremType1}
    Let $R$ be a trinomial algebra of Type~$1$.
    \begin{enumerate}
        \item [$(i)$] Any $\mathbb T$-homogeneous LND $\delta$ of $R$ is of the form: $$\delta=h{\partial\over\partial S_k} \qquad\textit{or}\qquad\delta=h\delta_C$$ 
    for some ${\TT}$-homogeneous $h$ in $\ker\partial/\partial S_k$ or respectively $\ker\delta_C$.
        \item [$(ii)$] Any ${\TT}$-homogeneous element $h\in\ker\partial/\partial S_k$ is of the form:     
        $$h=F\prod T_{ij}^{\alpha_{ij}}\prod\limits_{p\neq k} S_p^{\gamma_p};$$
        here $F$ is a polynomial in ${T_{i}^{l_{i}/d_i}}, \,i=1,\ldots,r$, $d_i=\gcd(l_{i1},\ldots,l_{in_i})$, and $\alpha_{ij},\,\gamma_p\in\ZZ_{\ge0}$.
        \item [$(iii)$] Any ${\TT}$-homogeneous element $h\in\ker\delta_C$ is of the form:
$$
    h=\lambda\prod\limits_{j\neq c_i} T_{ij}^{\alpha_{ij}}\prod S_p^{\gamma_p};
    $$
    here $\lambda\in\kk$ and $\alpha_{ij},\gamma_p\in\mathbb Z_{\ge0}$.
    \end{enumerate}        
\end{theorem}
The proof of Theorem~\ref{TheoremType1} is given after one preliminary result, which will be useful also in the next section; note that in the following lemma, we do not require that $\delta$ is ${\TT}$-homogeneous.



\begin{lemma}\label{PropositionGoodDerivationType1}
    Let $R$ be a trinomial algebra of Type~1 and $\delta\neq 0$ an LND of R possessing the following properties:
    \begin{enumerate}
        \item [$(a)$] $\delta(S_k)=0$ for all $k$;
        \item[$(b)$] there exists a tuple $C=(c_1,\ldots,c_r)$ such that $\delta(T_{ij})=0$ whenever $j\neq c_i$.
    \end{enumerate}
    Then $\delta=h\delta_C$ for a suitable $h\in\ker\delta_C$.
\end{lemma}
\begin{proof}
    For each pair $i, j$, the equation
    $
    T_i^{l_i}-T_j^{l_j}+a_i-a_j=0
    $
    holds. Applying $\delta$, we obtain:
    \begin{equation}\label{EqDeltaType1}
        \delta(T_{ic_i}){\partial T_i^{l_i}\over\partial T_{ic_i}}=\delta(T_{jc_j}){\partial T_j^{l_j}\over\partial T_{jc_j}}.
        \tag{2.1}
    \end{equation}    
    We claim that $\delta(T_{ic_i})\neq 0$ for all $i$. In fact, suppose that $\delta(T_{ic_i})=0$ for some $i$; then~\eqref{EqDeltaType1} implies that $\delta(T_{jc_j})=0$ for all $j$. But this means $\delta=0$, a contradiction.
    
    Now we claim that $l_{ic_i}=1$ for all $i$ except maybe one $i_0$. Assume the converse: $l_{ic_i}>1$ and $l_{jc_j}>1$ for two $i,j$. Then $T_{ic_i}$ divides the left-hand side of~(\ref{EqDeltaType1}), and $T_{jc_j}$ divides the right-hand side. Since the variables $T_{pq}$ are pairwise coprime (Lemma~\ref{LemmaCoprime}), we have
    $   T_{ic_i}\,|\,\delta(T_{j{c_j}})$ and $ T_{jc_j}\,|\,\delta(T_{ic_i}).
    $
    Then Lemma~\ref{BasicLem}$(iv)$ asserts that $\delta(T_{ic_i})=0$ or $\delta(T_{jc_j})=0$, contradicting our previous argument.

    We have verified that the tuple~$C$ satisfies the assumption of Construction~\ref{ConstrDer1}, i.e., the derivation~$\delta_C$ is well-defined. By~(\ref{EqDeltaType1}) we also have, for any pair $i\neq j$: $${\partial T_j^{l_j}\over\partial T_{jc_j}}\mid\delta(T_{ic_i}).$$ Then coprimeness of the variables (Lemma~\ref{LemmaCoprime}) implies that 
    $$\delta(T_{ic_{i}})=h_i\prod_{k\neq {i}} {\partial T_k^{l_k}\over\partial T_{kc_k}},
    \qquad i=1,\ldots,r.
    $$  
    Using~(\ref{EqDeltaType1}) again, we obtain $h_i=h_j=:h$ for all $i,j$. Thus $\delta=h\delta_C$. It follows that $h\in\ker\delta$ by Lemma~\ref{BasicLem}$(i)$.
\end{proof}
\begin{proof}[Proof of Theorem~\ref{TheoremType1}]
    \emph{Part $(i)$}. Let $\delta$ be a nonzero $\mathbb T$-homogeneous LND of $R$. Lemma~\ref{LemmaHLND}$(i)$ leaves two opportunities: 
    \begin{equation*}
        (a)\quad\delta(T_{ij})=0 \mbox{ for all $i,j$}\qquad\mbox{or} \qquad(b)\quad\delta(S_k)=0 \mbox{  for all $k$.}
    \end{equation*}
    In case $(a)$, it follows from Lemma~\ref{LemmaHLND}$(ii)$ that $\delta$ is a replica of some $\partial/\partial S_k$. 
    
    Consider case $(b)$. By Lemma~\ref{LemmaHLND}$(iii)$, there exists a tuple $C$ such that $T_{ij}\in\ker\delta$ whenever $j\neq c_i$. Then Lemma~\ref{PropositionGoodDerivationType1} asserts that $\delta=h\delta_C$ for a suitable $h\in\ker\delta_C$. Since both $\delta$ and $\delta_C$ are homogeneous, it follows that $h$ is. This concludes the proof of part $(i)$.

    \emph{Parts $(ii)$ and $(iii)$.} Any monomial of weight zero with respect to ${\TT}$ is easily seen to be a product of some $T_{i}^{l_{i}/d_i}$. It follows that each ${\TT}$-homogeneous element $h$ has the form $$h=F\prod T_{ij}^{\alpha_{ij}}\prod S_p^{\gamma_p},$$ where $F$ is a polynomial in $T_{i}^{l_{i}/d_i}, \,i=1,\ldots,r$. The kernel of an LND is factorially closed (Lemma~\ref{BasicLem}). Thus we have $h\in\ker{\partial/\partial S_k}$ if and only if $\gamma_k=0$. This proves part $(ii)$.

    In the same way, since $\ker\delta_C$ is factorially closed, we conclude that $h\in\ker\delta_C$ if and only if all $\alpha_{ic_i}=0$ and $\delta_C(F)=0$. Let us rewrite the latter condition on $F$ more explicitly. According to Construction~\ref{ConstrDer1}, there exists an $i_0$ such that $l_{ic_i}=1$ for $i\neq i_0$. We may assume that $i_0=1$, i.e., $T_{i}^{l_{i}/d_i}=T_i^{l_i}$ whenever $i\neq 1$. By relations~\eqref{equations1}, each $T_{i}^{l_i}$ is the sum of $T_{1}^{l_{1}}$ and a constant. Hence $F$ is a polynomial in $T_{1}^{l_{1}/d_1}$. Using the Leibniz rule, we get:
    $$
    0=\delta_C(F) = \delta_C\left(T_{1}^{l_{1}/d_1}\,\right) F',
    $$
    $F'$ denoting the derivative of $F$ with respect to $T_{1}^{l_{1}/d_1}$. By construction, $\delta_C(T_{1}^{l_{1}/d_1})\neq0$. Thus we have $\delta_C(F)=0$ if and only if $F=\mathrm{const}$.
\end{proof}

Theorem~\ref{TheoremType1} provides another proof of the rigidity criterion for trinomial algebras of Type~$1$, stated and proved in a different way in~\cite[Theorem 3]{G}.
\begin{corollary}\label{CorRigidiryType1}
    Let $R$ be a trinomial algebras of Type~$1$. Then $R$ is not rigid if and only if one of the following two conditions holds:
    \begin{enumerate}
        \item[$(a)$] $d>0$, i.e., there exists $S_k\in R$;
        \item[$(b)$] One can find an $i_0\in\{ 1,\ldots,r \}$ such that, for all $i\neq i_0$,  there exist $c_i\in\{1,\ldots,n_i\}$ with $l_{ic_i}=1$;
    \end{enumerate}
\end{corollary}
\begin{proof}
    By Lemma~\ref{LemmaLNDHomDecomposition}, $R$ is rigid if and only if it admits no ${\TT}$-homogeneous LND. The assertion follows.    
\end{proof}

\section{Semirigid trinomial algebras of Type~$1$}
Here we apply Theorem~\ref{TheoremType1} to obtain a criterion for a trinomial algebra of Type~$1$ to be semirigid, following the idea of~\cite[Section~6]{G_Danielewski}. Recall that the \emph{Makar-Limanov invariant} of an algebra $R$, abbreviated $\ML(R)$, is the intersection of the kernels of all LNDs of~$R$. 
\begin{lemma}\label{LemmaMLweek}
    Let $R$ be a trinomial algebra of Type~$1$. Suppose that the following conditions hold:
    \begin{enumerate}
        \item [$(a)$] $d=0$, i.e., one has no free variable $S_k\in R$
        \item [$(b)$] There exists an $i_0$ such that $T_{i_0}^{l_{i_0}}=T_{i_01}^{l_{i_01}}$ and $l_{i_01}>1$.
        \item [$(c)$] For each $i\neq i_0$, there exists a $c_i$ such that $l_{ic_i}=1$ and $l_{ij}>1$ whenever $j\neq c_i$.
    \end{enumerate}
    Then $\ML (R)=\kk[\,T_{ij}\mid i\neq i_0, \,j\neq c_i\,].$
\end{lemma}
\begin{proof}
      Theorem~\ref{TheoremType1} asserts that any ${\TT}$-homogeneous LND of $R$ is a replica of the $\delta_C$, $C=(c_1,\ldots,c_{i_0-1},1,c_{i_0+1},\ldots, c_r)$. Thus we have: $$     \ML (R)\subseteq\ker\delta_C=\kk[\,T_{ij}\,|\,i\neq i_0, \,j\neq c_i\,].
     $$
     
     To get the reverse inclusion, fix a variable $T_{ij}\in\ker\delta_C$. We use Remark~\ref{RemarkAct}, putting $b_i=c_i$. Then $\deg T_{ij}=\alpha e_{ij}$, $\alpha>0$, and $\deg\delta$ has nonzero coefficient at $e_{ij}$, for any ${\TT}$-homogeneous LND~$\delta$. Let $\mathbb Z^{n+d-r}$ be ordered in such a way that, for all $a,b\in\mathbb Z^{n+d-r}$, the inequality $a\le b$ implies the same inequality on the coefficients at $e_{ij}$. 
     
     By Lemma~\ref{LemmaLNDHomDecomposition}, any LND $\partial$ admits the decomposition
    $
    \partial=\sum_{k=1}^p\partial_k$; here $\deg\partial_1<\ldots<\deg\partial_p,
    $
    all the $\partial_k$ are ${\TT}$-homogeneous, and $\partial_1$ is locally nilpotent. 
    Hence, for any $k$, $\deg\partial_k$ has positive coefficient at $e_{ij}$, and the same holds for $\deg\partial_k(T_{ij})$. But $T_{ij}$ is the only variable whose weight has positive coefficient at $e_{ij}$. It follows that $T_{ij}\mid\partial_k(T_{ij})$ for all $k$, whence $T_{ij}\mid\partial(T_{ij})$. Now $\partial(T_{ij})=0$ by Lemma~\ref{BasicLem}. 
    
    Thus we obtain $\ker\delta_C\subseteq\ker\partial$ for any LND~$\partial$. It follows that $\ML(R)=\ker\delta_C$.
\end{proof}
\begin{corollary}\label{CorSemiRidigType1}
    A trinomial algebra $R$ of Type~1 is semirigid if and only if one of the following conditions holds:
    \begin{enumerate}
        \item[$(a)$] $R$ is rigid;
        \item[$(b)$] $d=1$ and the trinomial algebra $R_1$ obtained from $R$ by excluding $S_1$ from the system of generators, is rigid;
        \item[$(c)$] The assumption of Lemma~\emph{\ref{LemmaMLweek}} is fulfilled.      
    \end{enumerate}
\end{corollary}
\begin{proof}
      Clearly, if $R$ is semirigid, then $d\le1$; otherwise we would have non-equivalent LNDs ${\partial/\partial S_1}$ and ${\partial/\partial S_2}$.
    
    \emph{Case $d=1$}. We have $R=R_1\otimes\kk[S_1]$ where $R_1$ does not contain the free variables $S_k$.  Suppose that $R$ is semirigid; then all ${\TT}$-homogeneous LNDs are replicas of~${\partial/\partial S_1}$. Hence no tuple $C$ satisfies the assumption of Construction~\ref{ConstrDer1}, so Corollary~\ref{CorRigidiryType1} implies that $R_1$ is rigid. Thus we obtain~$(b)$.
    Conversely, if $R_1$ is rigid, semirigidity of $R$ follows from~\cite[Theorem~2.24]{Freu}.
    
    \emph{Case $d=0$}. We clearly get the assertion when $R$ is rigid. Now suppose that $R$ is semirigid but not rigid. Then there exists a unique $\delta_C$, which immediately implies~$(c)$. Conversely, suppose we have~$(c)$. By Lemma~\ref{LemmaMLweek}, any LND $\partial$ satisfies $\ker\partial\supseteq\kk[\,T_{ij}\,|\,i\neq i_0, \,j\neq c_i\,]$. Hence $\partial=h\delta_C$ by Lemma~\ref{PropositionGoodDerivationType1}, so $R$ is semirigid.
\end{proof}

\section{Main result on trinomial algebras of Type~2}
In this section, we state our main results on a trinomial algebra $R$ of Type~2. The proofs are given in subsequent sections and consist of two steps. In Section~\ref{SectionCases} we obtain formulae for ${\TT}$-homogeneous LNDs when $\Spec R$ is a surface. Then in Section~\ref{SectionProofs} we reduce Theorems~\ref{TheoremType2Classes} and~\ref{TheoremType2} to the surface case.

The following theorem describes the equivalence classes of ${\TT}$-homogeneous LNDs. Recall that two LNDs $\delta_1$ and $\delta_2$ are called \emph{equivalent} if $\ker\delta_1=\ker\delta_2$.

\begin{theorem}\label{TheoremType2Classes}
    Let $R$ be a trinomial algebra of Type~$2$. 
    \begin{enumerate}
        \item [$(i)$] Any ${\TT}$-homogeneous LND of $R$ either is a replica of ${\partial/\partial S_k}$, or corresponds to a tuple $C=(c_1,\ldots,c_r)$, $c_i\in\{1,\ldots,n_i\}$ such that $\delta(T_{ij})=0$ whenever $j\neq c_i$. In the latter case, $\delta(S_k)=0$ for all $k$ and one of the following conditions holds:
        \begin{enumerate}
            \item [$(a)$] There exist two integers $i_1,i_2$ such that $l_{ic_i}=1$ for all $i\neq i_1, i_2$.
            \item [$(b)$] There exist exactly three integers $i_1, i_2, i_3$ such that $l_{ic_i}=1$ for $i\neq i_1, i_2, i_3$. Furthermore, for $i=i_1,i_2$, all $l_{ij}$ are even and $l_{ic_i}=2$.
        \end{enumerate}
        \item [$(ii)$] Suppose that a tuple $C$ satisfies $(a)$ in $(i)$. If there exists $k\in\{i_1, i_2\}$ such that $l_{kc_k}\mid l_{ij}$ for  $i=i_1, i_2$ and all $j$, then the tuple $C$ corresponds to an infinite family of equivalence classes of ${\TT}$-homogeneous LNDs. Otherwise $C$ corresponds to exactly two classes of ${\TT}$-homogeneous LNDs.
        \item [$(iii)$] Suppose that $C$ satisfies \emph{$(b)$} in \emph{$(i)$}. If $l_{i_3 c_{i_3}}=2$ and all $l_{i_3 j}$ are even then $C$ corresponds to an infinite family of ${\TT}$-homogeneous LNDs. Otherwise $C$ corresponds to exactly two classes of ${\TT}$-homogeneous LNDs.
    \end{enumerate}
\end{theorem}

A straightforward corollary of Theorem~\ref{TheoremType2Classes} 
and~\cite[Theorem~2.24]{Freu} is the following:
\begin{corollary}\label{CorSemiRigidType2}
    A trinomial algebra $R$ of Type~$2$ is semi-rigid if and only if either $R$ is rigid, or there exists a unique variable $S_1$ and the algebra~$R_1$ is rigid, where $R_1$ is obtained from $R$ by removing $S_1$ from the set of generators.
\end{corollary}

The following construction provides explicit formulae for the ${\TT}$-homogeneous LNDs when $R$ is given by a single equation. It is convenient to assume that all coefficients in the equation equal~$1$; this can always be achieved by a linear change of variables. A square root of $-1$ is denoted by $\ii$.
\begin{construction}\label{ConstrDer2}
    Consider the trinomial variety given by $T_0^{l_0}+T_1^{l_1} + T_{2}^{l_2}=0$. Let $C=(c_0, c_1, c_2)$ with $c_i\in\{1,\ldots,n_i\}$. 
    
   Recall that $T_i^{l_i/\alpha}$ denotes $\prod_{j=1}^{n_i} T_{ij}^{l_{ij}/\alpha}$ whenever $\alpha\mid l_{ij}$ for all $j$.
    \begin{enumerate}
        \item [$(a)$] Suppose $l_{2c_2}=1$. Let 
                $$ \delta(T_{0c_0}) = 
                 \frac {\partial T_2^{l_2}}{\partial T_{2c_2}},\qquad \qquad
                 \delta(T_{2c_2})=-\frac {\partial T_0^{l_0}}{\partial T_{0c_0}},$$
                 $$\delta(T_{ij})=0, \quad\mbox{ other $i,j$},\qquad\quad 
                 \delta(S_k) = 0,\quad \text{all } k.$$
        \item [$(b)$] Suppose $l_{2c_2}=1$ and $l_{0c_0}\mid l_{ij}$ for $i=0,1$, $j=1,\ldots,n_i$. Take any $\lambda\in\Bbbk^\times$ and let
        $$ \delta(T_{0c_0}) = 
                 \frac {\partial T_1^{l_1/l_{0c_0}}}{\partial T_{1c_1}}{\partial T_2^{l_2}\over\partial T_{2c_2}},\qquad \qquad
                 \delta(T_{1c_1})=\lambda\frac {\partial T_0^{l_0/l_{0c_0}}}{\partial T_{0c_0}}{\partial T_2^{l_2}\over\partial T_{2c_2}},$$
                 $$
                 \delta(T_{2c_2})= -\left( {\partial T_0^{l_0}\over\partial T_{01}}{\partial T_1^{l_1/l_{0c_0}}\over\partial T_{11}} +\lambda {\partial T_0^{l_0/l_{0c_0}}\over\partial T_{01}} {\partial T_1^{l_1}\over\partial T_{11}}\right),
                 $$
        $$\delta(T_{ij})=0,\quad\mbox{ other $i,j,$}\qquad\quad\delta(S_k)=0,\quad\mbox { all $k$.}$$
        \item[$(c)$] Suppose $l_{0c_0}=l_{1c_1}=2$, $l_{2c_2}\ge2$, and all $l_{ij}$ are even for $i=0,1$. Take $\mu\in\{\pm\ii\}$ and let
        $$
        \delta(T_{0c_0})=\mu {\partial T_1^{l_1/2}\over\partial T_{1c_1}}{\partial T_2^{l_2}\over\partial T_{2c_2}},\qquad\qquad         
        \delta(T_{1c_1})={\partial T_0^{l_0/2}\over\partial T_{0c_0}}{\partial T_2^{l_2}\over\partial T_{2c_2}},
        $$              
        $$
        \delta(T_{2c_2})=-\left(\mu{\partial T_0^{l_0}\over\partial T_{0c_0}}{\partial T_1^{l_1/2}\over\partial T_{1c_1}} + {\partial T_0^{l_0/2}\over\partial T_{0c_0}}{\partial T_1^{l_1}\over\partial T_{1c_1}} \right).
        $$
        $$\delta(T_{ij})=0, \quad\mbox{ other $i,j$},\qquad\quad 
                 \delta(S_k) = 0,\quad \text{ all } k.$$
        \item[$(d)$] Suppose $l_{ic_i}=2$ for any $i$ and all $l_{ij}$ are even. Take $\lambda\in\Bbbk^\times$ and let
        \begin{equation*}
            \begin{split}                 \delta(T_{0c_0})&= \prod_{i\neq 0}{\partial T_i^{l_i/2}\over \partial T_{ic_i}}  \left(2\lambda\, T_1^{l_1/2} + (1+\lambda^2)\ii\, T_2^{l_2/2}\right) ,\\ 
            \delta(T_{1c_1})&= \prod_{i\neq 1}{\partial T_i^{l_i/2}\over \partial T_{ic_i}} \left(-2\lambda\, T_0^{l_0/2} + (1-\lambda^2) \,T_2^{l_2/2} \right) ,\\
            \delta(T_{2c_2})&= \prod_{i\neq 2}{\partial T_i^{l_i/2}\over \partial T_{ic_i}} \left(-(1+\lambda^2)\ii\, T_0^{l_0/2} - (1-\lambda^2) \,T_1^{l_1/2} \right).
            \end{split}
        \end{equation*}
       $$\delta(T_{ij})=0, \quad\mbox{ other $i,j$},\qquad\quad 
                 \delta(S_k) = 0,\quad \text{ all } k.$$
    \end{enumerate}
\end{construction}
The assignments just described give rise to derivations of the polynomial algebra by the Leibniz rule. It will follow from the proof of Theorem~\ref{TheoremType2} that these formulae indeed give ${\TT}$-homogeneous LNDs of the trinomial algebra.

\begin{theorem}\label{TheoremType2}
    Let $R$ be the trinomial algebra of Type~$2$ given by a single equation $T_0^{l_0}+T_1^{l_1}+ T_{2}^{l_2}=0$. Any ${\TT}$-homogeneous LND of $R$ is a replica either of ${\partial/\partial S_k}$ or of one of the derivations described in Construction~\ref{ConstrDer2} (after a permutation of the monomials~$T_{i}^{l_i}$, if necessary).

    Moreover, in the notation of Construction~\ref{ConstrDer2}, the algebra $\ker\delta$ is generated by $T_{ij}$, $i=0,1,2$, $j\neq c_i$, together with the following elements:
    \begin{align*}
        T_{1c_1}\qquad &\mbox{ in Case $(a)$},\\
        \lambda\,{ T_0^{l_0/l_{0c_0}}}-{T_{1}^{l_1/l_{0c_0}}}\qquad &\mbox{ in Case $(b)$},\\
        \mu\,{ T_0^{l_0/2}}+{T_{1}^{l_1/2}}\qquad  &\mbox{ in Case $(c)$},\\
        (1-\lambda^2)\, T_{0}^{l_0/2} -(1+\lambda^2)\ii\,{ T_1^{l_1/2}} +2\lambda\, {T_{2}^{l_2/2}}\qquad &\mbox{ in Case $(d)$}.
    \end{align*}
\end{theorem}

\begin{remark}
Elementary (i.e., $\HH$-homogeneous) LNDs correspond to Case $(a)$ and Case~$(b)$, the latter with $l_{0c_0}=1$; compare~\cite[Construction 2]{GZ}.
\end{remark}

\section{Non-rigid trinomial surfaces of Type 2}\label{SectionCases}
According to the rigidity criterion~\cite[Theorem~4]{EGSh}, a non-rigid trinomial algebra of Type $2$ with three generators $x, y, z$ is given by one of the following relations:
    \begin{enumerate}
        \item [$(a)$] $x^\alpha+y^\beta+z=0$ where $\alpha,\beta\in\ZZ_{>0}$;
        \item [$(b)$]$x^2+y^2+z^\gamma=0$ where  $\gamma\in\ZZ_{\ge2}$.
    \end{enumerate}
   In fact, \emph{any} non-rigid trinomial surface of Type 2 can be given by a single equation of the form $(a)$ or $(b)$ after eliminating linear terms. This section provides an explicit description of the ${\TT}$-homogeneous LNDs on such a surface.
\subsection{Case $x^\alpha+y^\beta+z=0$}\label{SectionCasePlane}
Let $X=\{x^\alpha+y^\beta+z=0\}$ where $\alpha,\beta\in\ZZ_{>0}$. Suppose $\alpha=a\gcd(\alpha,\beta),\,\beta=b\gcd(\alpha,\beta)$. Then~$X$ is equivariantly isomorphic to the affine plane with coordinate functions $x,y$ acted on by the one-dimensional torus ${\TT}$ via $t\colon\, x\longmapsto t^b x,\,\,y\longmapsto t^a y$.

\begin{definition}
    A polynomial $v\in\Bbbk[x_1,\ldots,x_n]$ is called a \emph{variable} if there exist $n-1$ polynomials $u_1,\ldots,u_{n-1}\in \Bbbk[x_1,\ldots,x_n]$ that, along with $v$, generate $\Bbbk[x_1,\ldots,x_n]$.
\end{definition}
\begin{lemma}\emph{\cite[Corollary~4.6]{Freu}}\label{A^2-Lemma}
    A derivation $\delta$ of the polynomial algebra $\Bbbk[x,y]$ is an LND if and only if there exists a variable $v$ and a polynomial $h\in\Bbbk[v]$ such that $$\delta\,\colon x\longmapsto {\partial h\over\partial y},\qquad y\longmapsto-{\partial h\over\partial x}.$$
\end{lemma}
\begin{proposition}\label{PropHom-A^2}
    Consider $\mathbb A^2$ with coordinate functions $x, y$ acted on by the one-dimensional torus ${\TT}$ via $t\colon\, x\longmapsto t^b x,\,\,y\longmapsto t^a y$ where $a, b\in\mathbb Z_{>0}$. Let $\delta$ be a ${\TT}$-homogeneous LND of $\Bbbk[x,y]$. Then one of the following holds:
    \begin{enumerate}
        \item [$(i)$] $\delta$ is a replica of ${\partial/\partial x}$;
        \item [$(ii)$] $\delta$ is a replica of ${\partial/\partial y}$; 
        \item[$(iii)$] $a=1$ and $\delta$ is a replica of $$x\lmt by^{b-1},\qquad  y\lmt\lambda,\qquad \lambda\in\Bbbk^{\times}$$ where  $\ker\delta=\Bbbk[\lambda x-y^{b}]$;
        \item[$(iv)$] $b=1$ and $\delta$ is a replica of  $$x\lmt \lambda,\qquad y\lmt ax^{a-1}, \qquad \lambda\in\Bbbk^{\times}$$ where $\ker\delta=\Bbbk[ x^{a}-\lambda y]$.
    \end{enumerate}
\end{proposition}
\begin{proof} 
    We make use of Lemma~\ref{A^2-Lemma}. The elements $$\delta(x)={ {dh}\over{dv}}{\partial v\over\partial y} \quad \mbox{and} \quad \delta(y)=-{dh\over {dv}}{\partial v\over\partial x}$$ are ${\TT}$-semi-invariant. Since the character group of ${\TT}$ is torsion-free, ${\partial v/\partial y}$ and ${\partial v/\partial x}$ are also ${\TT}$-semi-invariant; see, for example~\cite[Lemma 2.3\,(iii)]{A}. It follows in turn that $v$ is ${\TT}$-semi-invariant; hence $v=x^py^q\sum\limits_i \lambda_i\left(x^a/ y^b\right)^i$ where $p,q\in\ZZ_{\ge0},\,\lambda_i\in\Bbbk$. Furthermore, $h$ is a monomial in~$v$. As $v$ is irreducible (being a variable), the possibilities are:
    $v=\lambda_1 x$, $v=\lambda_2 y$ or $v=\lambda_1 x^a+\lambda_2 y^b$ where $\lambda_i\neq0$. 
    The first two cases lead directly to statements $(i)$ and $(ii)$. 
    
    Consider the case $v=\lambda_1 x^a+\lambda_2 y^b$ where $\lambda_i\neq0$. An element $v$ of this form can be included in a transcendence basis of $\Bbbk[x,y]$ if and only if $a=1$ or $b=1$. Suppose for instance that $a=1$. We also may assume that $v= \lambda x- y^b$ and $h$ is proportional (over $\Bbbk$) to $v^q$, $q\in \ZZ_{>0}$. 
    Thus $\delta$ is proportional over $\Bbbk$ to the derivation
    $
    x\longmapsto b y^{b-1} v^{q-1}$, $y\longmapsto  \lambda v^{q-1},
    $ 
    a replica of the one in~$(iii)$. Then $\ker\delta=\kk[v]$ provided by Lemma~\ref{BasicLem}\,$(i)$.
    
    Case $b=1$ leads to statement~$(iv)$ in the same manner.
\end{proof}

\subsection{Case $x^2+y^2+z^\gamma=0$, $\gamma>2$}\label{CaseKotenkova}
Let the surface given by this equation be denoted by $X$. For $\gamma$ odd, the one-dimensional torus ${\TT}$ acts by the rule $t\colon\,x\longmapsto t^\gamma x,$ $y\longmapsto t^\gamma y,$ $z\longmapsto t^2z$ where $t\in {\TT}$. If $\gamma$ is even, the action is given by $t\colon\,x\longmapsto t^{\gamma/2} x,$ $y\longmapsto t^{\gamma/2} y,$ $z\longmapsto tz$.

Consider the toric variety $\{uv=z^\gamma\}$ where the torus $(\Bbbk^\times)^2$ acts via $(t_1,\,t_2)\colon\,u \lmt t_1^\gamma t_2^{-1}u$, $v\lmt t_2 v$, $z\lmt t_1 z$.  Under the linear change of variables $u=\ii x- y,\,v=\ii x+ y$, this variety is isomorphic to $X$. Moreover, the actions of ${\TT}$ and $(\kk^\times)^2$ are compatible via the embedding $\varphi\colon\,{\TT}\hookrightarrow(\Bbbk^\times)^2$, where $\varphi(t)=(t^2,\,t^\gamma)$ for $\gamma$ odd and $\varphi(t)=(t,\,t^{\gamma/ 2})$ for $\gamma$ even. Demazure roots, introduced in~\cite{D}, provide a well-known description of homogeneous LNDs in the toric case. We recall it to the extent needed to describe the ${\TT}$-homogeneous LNDs, since they all turn out to be $(\kk^\times)^2$-homogeneous as well.

\begin{notations}(see~\cite[Chapter~1]{CLS} for details)
\begin{enumerate}
    \item [$(i)$] With the torus $(\Bbbk^\times)^2$ are associated the lattice of one-parameter subgroups $N=\Hom(\Bbbk^\times,\,(\Bbbk^\times)^2)\simeq \ZZ^2$, which is dual to the character lattice $M=\Hom((\Bbbk^\times)^2,\,\Bbbk^{\times})\simeq \ZZ^2$. The corresponding $\QQ$-vector spaces are denoted by $N_{\QQ}=N\otimes_{\ZZ} \QQ$ and $M_{\QQ}=M\otimes_{\ZZ} \QQ$.
The toric variety $X$ corresponds to a rational polyhedral cone $\sigma_X\subseteq N_{\QQ}$ and its dual $\sigma_X^\vee\subseteq M_{\QQ}$, so that $\Bbbk[X]=\Bbbk[\sigma_X^\vee\cap M]$.
    \item [$(ii)$] The character corresponding to a lattice vector $m\in M$ is denoted by $\chi^m$.
    \item [$(iii)$] If $\rho$ is a ray of a cone $\sigma\subseteq N$, then $n_\rho$ denotes the first lattice vector on the ray $\rho$.
    \item [$(iv)$] There is a natural pairing $\langle\,\cdot\,,\,\cdot\,\rangle\,\colon\,N\times M\to\ZZ$.
\end{enumerate}
\end{notations}

\begin{definition}
    Let $N, M$ be dual lattices and $\sigma\subseteq N_{\QQ}$ a polyhedral cone. A \emph{Demazure root} associated to the ray $\rho$ of~$\sigma$ is defined as an $e\in M$ such that $\langle n_\rho,\,e\rangle=-1$ and $\langle n_{\rho'},\, e\rangle\ge0$ for all rays $\rho'\neq\rho$ of the cone $\sigma$. The set of Demazure roots corresponding to the ray $\rho$ will be denoted by~$S_\rho$.
\end{definition}

To describe the homogeneous LNDs on $X$, we are interested in $\sigma=\sigma_X$.

\begin{proposition}\label{PropHomCaseGamma}
    Any ${\TT}$-homogeneous LND on the variety $$X=\{x^2+y^2+z^\gamma=0\},\quad\mbox{ where $\gamma>2$},$$ is a replica of the derivation
    \begin{equation*}
        \delta_{0}\,\colon\,x\longmapsto \ii\gamma z^{\gamma-1},\quad y\longmapsto \gamma z^{\gamma-1},\quad z\longmapsto -2(\ii x+y),\qquad\ker\delta_0=\Bbbk[\ii x+y],
   \end{equation*}
   or
   $$
    \delta_{\infty}\,\colon\,x\longmapsto -\ii\gamma z^{\gamma-1},\quad y\longmapsto \gamma z^{\gamma-1},\quad z\longmapsto 2(\ii x-y),\qquad\ker\delta_\infty=\Bbbk[\ii x-y].
   $$
\end{proposition}
\begin{proof}
    Let us first show that all ${\TT}$-homogeneous LNDs on $X$ are $(\Bbbk^\times)^2$-homogeneous. The action $(\Bbbk^\times)^2\curvearrowright X$ corresponds to the cones $$\sigma_X^\vee=\cone((1,0),\,(1,\gamma))\subset M, \qquad \sigma_X=\cone((0,1),\, (\gamma, -1))\subset N.$$
    The embedding ${\TT}\subset(\Bbbk^\times)^2$ induces an embedding of lattices of one-parameter subgroups $\Hom(\Bbbk^\times, {\TT})\subset \Hom(\Bbbk^\times, (\Bbbk^\times)^2)=N$, whose image we denote by~$\Gamma_T$. It is easy to see that the line $\Gamma_T$ is generated by the vector $(2,\gamma)$ when $\gamma$ is odd, and by $(1, \gamma/2)$ when $\gamma$ is even. In both cases we are in the situation of item~3.2 of Proposition~6 in~\cite{Kotenkova}. Thus, all ${\TT}$-homogeneous LNDs are $(\Bbbk^\times)^2$-homogeneous.

     Recall that the algebra $\Bbbk[X]$ is generated by the characters $\chi^m\in\sigma_X^\vee\cap M$. In our case, one can take as generators $$\chi^{(\gamma,-1)}=u=\ii x-y, \quad\chi^{(0,1)}=v=\ii x+y, \quad\chi^{(1,0)}=z.$$ It follows from~\cite[Theorem~2.7]{Liendo} that each $(\Bbbk^\times)^2$-homogeneous LND is proportional over $\Bbbk$ to a derivation $\partial_{\rho, e}\,\colon\, \chi^m\longmapsto\langle m,\,\rho\rangle\,\chi^{m+e}$ for some ray $\rho$ of the cone $\sigma_X$ and some Demazure root $e\in S_\rho$. Let the ray $\rho_1$ be generated by the vector $(0,1)$ and $\rho_2$ by $(\gamma, -1)$. A straightforward computation shows that the Demazure roots are of the form
     $$
     S_{\rho_1}=\{\, (-1, p),\quad p\in\ZZ_{>0}\,\}\qquad\mbox{ and }\qquad S_{\rho_2}=\{\, (\gamma p-1, -p),\quad p\in\ZZ_{>0}\,\}.
     $$
     Consequently,
     $$\partial_{\rho_1,\, e}\,\colon\,u\lmt\gamma z^{\gamma-1}v^{p-1},\quad v\lmt0,\quad z\lmt v^p,\quad e=(-1, p),\quad p\in\ZZ_{>0}$$
     and
     $$\partial_{\rho_2,\, e}\,\colon\,u\lmt0,\quad v\lmt \gamma z^{\gamma-1}u^{p-1},\quad z\lmt u^p,\quad e=(\gamma p-1, -p),\quad p\in\ZZ_{>0}.$$
     Passing to the coordinates $x,y,z$, we obtain $$\partial_{\rho_1,e}=-2(\ii x+y)^{p-1}\delta_0\quad\mbox{ and }\quad \partial_{\rho_2, e}=2(\ii x-y)^{p-1}\delta_\infty,$$ where the $\delta_i$ are described in the statement.

     By Lemma~\ref{BasicLem}\,$(i)$, any ${\TT}$-homogeneous element $h\in\ker\delta_i$ is a monomial in $\ii x+y$ or $\ii x-y$, respectively. Since $\ker\delta_i$ is generated as a subalgebra by homogeneous elements, we obtain $\ker\delta_0=\Bbbk[\ii x+y]$ and $\ker\delta_\infty=\Bbbk[\ii x-y]$.
\end{proof}

\begin{remark}\label{RemarkCompareT&H}
    The algebra $\kk[X]$ with $X$ as in Proposition~\ref{PropHomCaseGamma}, admits two nontrivial equivalence classes of ${\TT}$-homogeneous LNDs. At the same time, there are no $\HH$-homogeneous LNDs of $\kk[X]$, see~\cite[Corollary~1]{R} or~\cite[Theorem~1]{GZ}. Another example of such a situation (where the number of equivalence classes of ${\TT}$-homogeneous LNDs is even infinite) can be found in Proposition~\ref{PropHomCase222}.

\end{remark}

\subsection{Case $x^2+y^2+z^2=0$}\label{PropHomCaseDivisor}
Let $X$ denote the trinomial surface given by this equation.
The one-dimensional torus ${\TT}$ acts by the rule $t\,\colon\, x\longmapsto tx,\,y\lmt ty,\,z\lmt tz$. Again we use the fact that the substitution $u=\ii x-y,\,v=\ii x+y$ turns $X$ into a toric variety with the action of the two-dimensional torus $(\Bbbk^\times)^2$ given by the rule $(t_1,\,t_2)\,\colon\,u\lmt t_1^2t_2^{-1}u,$ $v\lmt t_2v$, $z\lmt t_1z$. The equivariant embedding $\varphi\,\colon\,{\TT}\hookrightarrow(\Bbbk^\times)^2$ is given by the formula $t\lmt(t,\,t)$. Let $N_{\TT}$ and $N_X$ denote the lattices of one-parameter subgroups of the tori ${\TT}$ and $(\Bbbk^\times)^2$, respectively. The action $(\Bbbk^\times)^2\curvearrowright X$ corresponds to the cone $\sigma_X=\cone((1,0),\,(1,2))$, and the embedding $\varphi$ corresponds to the diagonal embedding of the lattices of one-parameter subgroups $N_{{\TT}}\hookrightarrow N_X$. In contrast to case~\ref{CaseKotenkova}, not all ${\TT}$-homogeneous LNDs are $(\Bbbk^\times)^2$-homogeneous. We therefore use the technique based on the notion of a proper polyhedral divisor, developed in~\cite{Altmann-Hausen, Liendo}. A brief description of this technique can be found in~\cite[Section~1]{AL}.

\begin{notations}
    Let $N$ and $M$ be dual lattices (in general, not related to $N_{\TT}$ and $N_X$). Let $\sigma\subseteq N_\QQ$ be a polyhedral cone, and let $\DD=\sum_{p\in Y} \Delta_p\cdot p$ be a proper $\sigma$-polyhedral divisor on a smooth semiprojective curve $Y$. For $m\in \sigma^\vee\cap M$, set
    \begin{equation*}
        \begin{split}
            \DD(m)&=\sum_{p\in Y} \min_{q\in\Delta_p}\langle q, m\rangle\cdot p,\\
            A_m&=H^{0}(Y,\, \OO(\DD(m)))= \{f\in\Bbbk(Y)\,|\,\divisor (f)+\DD(m)\ge0\},\\
            A[Y,\DD] &=\bigoplus_{m\in\sigma^\vee\cap M}A_m. 
        \end{split}
    \end{equation*}
    Note that $A[Y,\DD]$ has a natural structure of a $\kk$-algebra, and the $\sigma^\vee\cap M$-grading determines on it an action of the torus $N\otimes_{\ZZ} \Bbbk^\times$.
\end{notations}
\begin{lemma}\label{LemmaA[Z,D]}
There exists an equivariant isomorphism $\Bbbk[X]\simeq A[\PP^1,\,\DD]$, where $\DD=\QQ_{\ge 2}\{\infty\}$.
\end{lemma}
\begin{proof}
    By~\cite[Theorem 1.2]{AL}, there exists an equivariant isomorphism of algebras $\Bbbk[X]\simeq A[Y,\DD]$ for some semiprojective curve $Y$, cone $\sigma$ and proper $\sigma$-polyhedral divisor $\DD$ on $Y$. In~\cite[Section~11]{Altmann-Hausen}, it is described how to construct the pair $(Y,\,\DD)$ in the case when $X$ is an affine toric variety. Namely, the fan of the toric variety $Y$ is defined via the exact sequence
    \[
    \begin{tikzcd}
        0 \arrow[r] &N_{\TT} \arrow[r, "F"]  &N_X \arrow[l, bend left,  "S"] \arrow[r,"P"] &N_Y \arrow[r] &0,\\
    \end{tikzcd}
    \]
    where $N_Y:= N_X/N_{\TT}$, $F$ is the embedding of lattices of one-parameter subgroups
    and $S\circ F=\id$. In our case, $F$ is the diagonal embedding and $N_Y\simeq \ZZ$. Set also $P(a,b)=a-b$ and $S(a,b)=a$. The variety $Y$ is determined by the coarsest fan $\Sigma_Y$ each cone of which is contained in the image of a face of the cone $\sigma_X$. Therefore, $\Sigma_Y=\{\QQ_{\ge0}, \QQ_{\le0}, 0\}$, i.e., $Y\simeq\PP^1$. We also obtain $\DD=\QQ_{\ge1}(\{0\}+\{\infty\})$ and $\sigma=\QQ_{\ge0}$. One can replace $\DD$ by $\QQ_{\ge 2}\{\infty\}$, using~\cite[Corollary~1.3]{AL}.
\end{proof}
\begin{remark}\label{RemarkA[Z,D]}
    Let $\xi$ be a coordinate on $\A^1=\PP^1\setminus\{\infty\}$ and let $\chi$ be a generator of the character group of the torus ${\TT}$.  Then $A_m=\bigoplus_{k=0}^{2m}\Bbbk\xi^k$ and $$A[\PP^1,\,\DD]=\bigoplus_{m\in\ZZ_{\ge0}}\bigoplus_{k=0}^{2m}\Bbbk\chi^m\xi^k\simeq \Bbbk[u,v,z]/(uv-z^2),$$ where $u=\chi,\,v=\chi\xi^2,\,z=\chi\xi$.
\end{remark}

\begin{definition}
    A derivation $\delta$ of the algebra $\Bbbk[X]$ extends uniquely to a derivation of the field of fractions $\Bbbk(X)$. Let $\Bbbk(X)^{\TT}$ denote the field of ${\TT}$-invariant rational functions. A ${\TT}$-homogeneous LND $\delta$ is said to be of \emph{vertical type} if $\delta(\Bbbk(X)^{{\TT}})=0$, and of \emph{horizontal type} otherwise.
\end{definition}
\begin{proposition}\label{PropHomCase222}
    Any ${\TT}$-homogeneous LND on the variety $X=\{x^2+y^2+z^2=0\}$ is a replica of $\delta_\infty$ or of $\delta_\lambda$, $\lambda\in\kk$, where
    $$
    \delta_\infty\,\colon\,x\lmt -\ii z,\quad y\lmt z,\quad z\lmt \ii x-y,\qquad \ker\delta_\infty=\Bbbk[\ii x-y],
    $$
    and
    $$\delta_\lambda\,\colon\,x\lmt 2\lambda y+\ii (1+\lambda^2)z,\qquad y\lmt -2 \lambda x+(1-\lambda^2)z,$$
    $$z\lmt -\ii(1+\lambda^2)x-(1-\lambda^2)y,\qquad \ker\delta_\lambda=\Bbbk[(1-\lambda^2)x-\ii(1+\lambda^2)y + 2\lambda z].$$
\end{proposition}
\begin{proof}
    We use the representation $\Bbbk[X]\simeq A[\PP^1,\,\DD]$, where $\DD=\QQ_{\ge 2}\{\infty\}$, from Lemma~\ref{LemmaA[Z,D]} and Remark~\ref{RemarkA[Z,D]}. By~\cite[Lemma~2]{A_Rigid}, there are no ${\TT}$-homogeneous LNDs of vertical type on $\Bbbk[X]$. This fact also follows from~\cite[Theorem~1.7]{AL}, since the only Demazure root for~$\sigma$ is $e=-1$, and the sheaf $\OO(\DD(-1))$ has no nonzero global sections on~$\PP^1$.

    The ${\TT}$-homogeneous LNDs of horizontal type on an algebra of the form $A[\PP^1,\DD]$ correspond to coherent pairs $(\Tilde \DD,\, e)$, where $\Tilde\DD$ is a coloring of the divisor $\DD$; see~\cite[Definitions~1.8,~1.9 and Theorem~1.10]{AL}. The explicit form of the LND is given by formula~(2) in~\cite{AL}. In our case, any coloring of the divisor $\DD$ has the form $\Tilde\DD=\{\DD;\, v_p,\,p\in\PP^1\setminus \{p_\infty\}\}$ for some point $p_\infty\in\PP^1$. From the form of the divisor $\DD$, it follows that $v_p=0$ for $p\neq \infty$, and $v_\infty=2$.

    {\bf Case $p_\infty=\infty$.} The divisor $\DD$ restricts trivially to $\A^1=\PP^1\setminus\{p_\infty\}$, i.e., all coefficients of $\DD|_{\PP^1\setminus\{p_\infty\}}$ equal $\QQ_{\ge0}$. Thus, we have the cone $\tilde\omega=\cone((0,1),\,(2,-1))$ with distinguished ray $\tilde\rho$ spanned by the vector $(0,1)$. The set of Demazure roots has the form $$S_{\tilde\rho}=\{(e,-1)\mid e\in\ZZ_{\ge0}\}.$$ Since all coefficients of the divisor $\DD$ have exactly one vertex, which moreover has integer nonnegative coordinate, every pair $(\Tilde\DD, e),\,\,e\in\ZZ_{\ge0},$ is coherent. By formula~(2) in~\cite{AL}, the corresponding LNDs $\partial_e$ are of the form $\partial_e(\chi^m\xi^r)=r\chi^{m+e}\xi^{r-1}$. Consequently,
    \begin{equation*}
        \begin{split}
            \partial_e(u) &=\partial_e(\chi)=0,\\
            \partial_e(v) &= \partial_e(\chi\xi^2) = 2\chi^{1+e}\xi=2zu^e,\\
            \partial_e(z) &= \partial_e(\chi\xi) = \chi^{1+e}=u^{1+e}.
        \end{split}
    \end{equation*}  Thus, $\partial_e=u^e \delta_\infty$, where $\delta_\infty\,\colon\,u\lmt 0,$ $v\lmt 2z,$ $z\lmt u.$ By Lemma~\ref{BasicLem}$(i)$, $\ker\delta_\infty=\Bbbk[u]$. Passing to the coordinates $x,y,z$, we obtain
    $$
    \delta_\infty\,\colon\,x\lmt-\ii z,\quad y\lmt z,\quad z\lmt\ii x-y,\qquad \ker\delta_\infty=\Bbbk[\ii x-y].
    $$

    {\bf Case $p_\infty\neq\infty$.} Then $\DD|_{\PP^1\setminus\{p_\infty\}}=\QQ_{\ge2}\{\infty\}$, $\tilde\omega=\cone((2,1),\,(0,-1))$ and the distinguished ray $\tilde\rho$ is spanned by the vector $(2,1)$. The set of Demazure roots has the form
    $$S_{\tilde\rho}=\{(e,-1-2e)\mid e\in\ZZ_{\ge0}\}.$$

    Let $\lambda\in\Bbbk$ denote the $\xi$-coordinate of the point $p_\infty$. We need to choose a coordinate $\eta$ such that $p_\infty$ is given by the equality $\eta=\infty$. Set $\eta=(\xi-\lambda)^{-1}$, i.e., $\xi=\eta^{-1}+\lambda$, or $\eta^{-1}=\xi-\lambda$. Then formula~(2) of~\cite{AL} gives $$\partial_e(\chi^m\eta^r)=(2m+r)\chi^{m+e}\eta^{r-1-2e}.$$
    Let us express the variables $u,v,z$ in terms of $\chi, \eta$. We have $$u=\chi, \qquad v=\chi\xi^2=\chi\eta^{-2}+2\lambda\chi\eta^{-1}+\lambda^2\chi,\qquad z=\chi\xi=\chi\eta^{-1}+\lambda\chi.$$ Substituting into the formula for $\partial_e$, we obtain
    \begin{equation*}
        \begin{split}
            \partial_e(u) &=
    2(z-\lambda u)(v-2\lambda z+\lambda^2 u)^e,\\ \partial_e(v) &= 2\lambda(v-\lambda z)(v-2\lambda z+\lambda^2 u)^e,\\
    \partial_e(z) &= (v-\lambda^2 u)(v-2\lambda z+\lambda^2 u)^e.
        \end{split}
    \end{equation*}
    Thus, $\partial_e$ is a replica of the derivation $$\delta\,\colon\,u\lmt 2(z-\lambda u),\quad v\lmt 2\lambda(v-\lambda z),\qquad z\lmt v-\lambda^2 u$$ and $\ker\delta=\Bbbk[v-2\lambda z+\lambda^2 u]$ by Lemma~\ref{BasicLem}$(i)$. In coordinates $x,y,z$, we obtain $$\delta\colon\,x\lmt -2\ii\lambda y-\ii(1-\lambda^2)z,\qquad y\lmt 2\ii\lambda x-(1+\lambda^2)z,$$
    $$z\lmt \ii(1-\lambda^2)x+(1+\lambda^2)y,$$
    $\ker\delta=\Bbbk[\ii(1+\lambda^2)x+(1-\lambda^2)y - 2\lambda z].$
    It is easy to see that $\delta=-\delta_{\ii\lambda}$.
\end{proof}
\begin{remark}
    Note that $\delta_\infty$ and $\delta_0$ are proportional to the corresponding LNDs from Proposition~\ref{PropHomCaseGamma} at $\gamma=2$. Permutations of the variables $x,y,z$ take the LNDs $\delta_\lambda$, where $\lambda\in\{0,\infty,\pm1,\,\pm\ii\}$, to LNDs proportional to those in this set.
\end{remark}

\section{Proofs of Theorems~\ref{TheoremType2Classes} and~\ref{TheoremType2}}\label{SectionProofs}

\begin{proof}[Proof of Theorem~\ref{TheoremType2}]
    Let $\delta$ be a nonzero ${\TT}$-homogeneous LND on $R$. By Lemma~\ref{LemmaHLND}, it suffices to consider the case when the variables $S_k$ are absent. Then there exists a tuple $C$ such that $\delta(T_{ij})=0$ for $j\neq c_i$.
   Let $\KK$ be the algebraic closure of the field obtained from $\Bbbk$ by adjoining all variables $S_k$ and $T_{ij}$ with $j\neq c_i$.
    Consider the trinomial surface $Z$ over the field $\KK$, given by the same equation as $\Spec R$. The equation has the form $x^{l_{0c_0}}+y^{l_{1c_1}}+z^{l_{2c_2}}=0$ after the substitution
    \begin{equation*}
    \begin{split}
        x &= T_{01}\sqrt[l_{01}]{T_{02}^{l_{02}}\dots T_{0n_0}^{l_{0n_0}}},\\
        y &= T_{11}\sqrt[l_{11}]{T_{12}^{l_{12}}\dots T_{1n_1}^{l_{1n_1}}},\\
        z &= T_{21}\sqrt[l_{21}]{T_{22}^{l_{22}}\dots T_{2n_2}^{l_{2n_2}}}.
    \end{split}
    \end{equation*} 
    It is clear that $\delta$ descends to a nonzero LND on $\KK[Z]$, which is homogeneous with respect to the action of the one-dimensional torus $\KK^\times$ on $Z$. By the rigidity criterion~\cite[Theorem~4]{EGSh}, $Z$ is given (possibly after a permutation of variables) by one of the following equations:
    \begin{enumerate}
        \item [{\rm (A)}] $x^\alpha+y^\beta+z=0,$ where $\alpha,\beta\in\ZZ_{>0}$;
        \item [(B)] $x^2+y^2+z^\gamma=0$, where $\gamma\in\ZZ_{>2}$;
        \item [(C)] $x^2+y^2+z^2=0$.
    \end{enumerate}
    Thus, $\delta$ corresponds to a homogeneous LND on a surface of the form (A), (B) or (C). Conversely, a homogeneous LND $\hat\delta$ on $\KK[Z]$ lifts to a homogeneous LND on $R$ if and only if $\hat\delta(T_{ic_i})\in R$ for all $i$.

    \textbf{Case (A).} In this case the homogeneous LNDs are described in Section~\ref{SectionCasePlane}. It suffices to consider items $(i)$ and $(iii)$ of Proposition~\ref{PropHom-A^2}, since the remaining cases are obtained by permutation of variables.

    \textbf{(A1)}
    Suppose $\delta$ is a replica of ${\partial/\partial x}$. Then $\delta(T_{1c_1})=0$. Applying $\delta$ to the relation $T_0^{l_0}+T_1^{l_1}+ T_{2}^{l_2}=0$, we obtain
    $$
    \delta(T_{0c_0}){\partial T_0^{l_0}\over\partial T_{0c_0}}=-\delta(T_{2c_2}){\partial T_2^{l_2}\over\partial T_{2c_2}}.
    $$
    Since the variables $T_{ij}$ are pairwise coprime (Lemma~\ref{LemmaCoprime}), we have
    $$
    {\partial T_0^{l_0}\over\partial T_{0c_0}}\mid \delta(T_{2c_2}) \qquad\mbox{ and }\qquad {\partial T_2^{l_2}\over\partial T_{2c_2}}\mid \delta(T_{0c_0}).
    $$
    Thus, $\delta$ is a replica of the derivation $(a)$ from Construction~\ref{ConstrDer2}. Moreover, $$\ker\delta=\KK[y]\cap R=\kk[T_{1c_1},\,T_{ij}\mid j\neq c_i].$$

    \textbf{(A2)}
    Suppose now that $\delta$ is a replica of the derivation from item~$(iii)$ of Proposition~\ref{PropHom-A^2}. In particular, $\beta=b\alpha$, where $b\in\ZZ$. We will show that in this case $\alpha = \gcd(l_{01},\ldots,l_{0n_0})$ and $\alpha\mid\gcd(l_{11},\ldots,l_{1n_1})$.
    We have
    \begin{equation}\label{Eq0}
        \delta(x)=by^{b-1}f,\qquad\delta(y)=\lambda f,\qquad f\in\KK[\lambda x-y^b],\quad\lambda\in\KK.
        \tag{6.1}
    \end{equation} whence
    \begin{equation}\label{Eq3}
        {\delta(T_0^{l_0})\over \delta(T_1^{l_1})}={\delta(x^\alpha)\over \delta(y^{\beta})} = {1\over\lambda}{x^{\alpha-1}\over y^{\beta-b}} ={1\over\lambda}{\sqrt[\alpha]{T_0^{l_0}}\over \sqrt[\alpha]{{T_1^{l_1}}}}\cdot{T_1^{l_1}\over T_0^{l_0}}.
        \tag{6.2}
    \end{equation}   
    By the homogeneity of $\delta$, the expression~(\ref{Eq3}) lies in the field of rational invariants $\Quot(R)^{\TT}$, each element of which by Lemma~\ref{LemmaRationalInvariants} is representable as a rational function over $\Bbbk$ in $T_i^{l_i/d_i}$, where $d_i=\gcd(l_{i1},\ldots,l_{in_i})$. Since $\lambda$ does not depend on $T_{01},\,T_{11}$, we obtain $\alpha\mid d_0$ and $\alpha\mid d_1$, as well as $\lambda\in\Bbbk$.

    Next, eliminating $f$ from the relations~(\ref{Eq0}):
    $$
    \lambda\cdot\delta(T_{01}) {\partial T_0^{l_0/\alpha}\over\partial T_{01}} = \delta(T_{11}) {\partial T_1^{l_1/\alpha}\over\partial T_{11}}.
    $$
    Now the coprimeness of the variables (Lemma~\ref{LemmaCoprime}) shows that
    $$
    \delta(T_{01})= h\cdot {\partial T_1^{l_1/\alpha}\over\partial T_{11}},\qquad \delta(T_{11})= \lambda h \cdot {\partial T_0^{l_0/\alpha}\over\partial T_{01}}
    $$
    for some $h\in R$. Applying $\delta$ to the relation $T_0^{l_0}+T_1^{l_1}+T_2^{l_2}=0$, we obtain
    \begin{equation}\label{Eq4}
        \delta(T_{21}){\partial T_2^{l_2}\over\partial T_{21}} =- h \left( {\partial T_0^{l_0}\over\partial T_{01}}{\partial T_1^{l_1/\alpha}\over\partial T_{11}} +\lambda {\partial T_0^{l_0/\alpha}\over\partial T_{01}} {\partial T_1^{l_1}\over\partial T_{11}}\right).
        \tag{6.3}
    \end{equation}
    Note that the monomial on the left-hand side and the expression in parentheses on the right-hand side are not, in general, coprime in $R$. We therefore need the following lemma.
    \begin{lemma}\label{LemmaQuasiCoprime}
        Let a monomial $P$ and polynomials $Q, F$ be such that $P\mid T_2^{l_2}$, $F$ contains no monomials divisible by~$T_2^{l_2}/P$, and $Q$ does not depend on any of the variables $T_{2j}$.  Suppose that $FP=hQ$ in $R$ for some $h\in R$. Then $P\mid h$ in $R$.
        \end{lemma}
        \begin{proof}
            Choose a polynomial $H$ representing $h\in R$ in such a way that it contains no monomials divisible by $T_2^{l_2}$. Then
            $
           FP-HQ\in I
            $
            and each term in $FP-HQ$ contains no monomials divisible by $T_2^{l_2}$. Consequently,
            $FP=HQ.$ Since the polynomials $P$ and $Q$ are coprime, we obtain $P\mid H$.
        \end{proof}
    It follows from relations~(\ref{Eq0}) that $\delta(T_{21})$ can be expressed as a polynomial not depending on~$T_{21}$.
    Applying Lemma~\ref{LemmaQuasiCoprime} to relation~(\ref{Eq4}), we conclude that $\delta$ is a replica of the derivation~$(b)$ from Construction~\ref{ConstrDer2}.
    Moreover, the subalgebra $\ker\delta$
    is generated over $\Bbbk$ by the variables $T_{ij}$, where $j\neq c_i$, along with the element $\lambda{ T_0^{l_0/\alpha}}-{T_{1}^{l_1/\alpha}}$.

    \textbf{Case (B).}
    The homogeneous LNDs on $Z$ are found in Proposition~\ref{PropHomCaseGamma}. Thus, $\delta$ is a replica of the derivation $\delta_{\infty}$ or $\delta_0$. Let us consider, for example, the case $\delta_0$. We have
    $$
    \delta_{0}\,\colon\,x\longmapsto \ii\gamma z^{\gamma-1},\quad y\longmapsto \gamma z^{\gamma-1},\quad z\longmapsto -2(\ii x+y),\quad\qquad\ker\delta_0=\KK[\ii x+y].
    $$
    By considering the ratio $\delta(T_0^{l_0})/\delta(T_1^{l_1}) = \ii x/y$, analogously to case~(A2), we conclude that $l_{ij}$ is even for $i=0,1$.

    We also have $\delta(x) = \ii\delta(y)$, that is, $$\delta(T_{01}){T_{02}^{l_{02}/2}\dots T_{0n_0}^{l_{0n_0}/2}}=\ii\delta(T_{11}){T_{12}^{l_{12}/2}\dots T_{1n_1}^{l_{1n_1}/2}},$$ whence, for some $h\in R$,
    $$
    \delta(T_{01})=\ii h{T_{12}^{l_{12}/2}\dots T_{1n_1}^{l_{1n_1}/2}}, \qquad \delta(T_{11}) = h{T_{02}^{l_{02}/2}\dots T_{0n_0}^{l_{0n_0}/2}}.
    $$
    Applying $\delta$ once more to the relation $T_0^{l_0}+T_1^{l_1}+T_2^{l_2}=0$, we obtain
    $$
    {\partial T_2^{l_2}\over \partial T_{21}}\delta(T_{21}) = -2h{T_{12}^{l_{12}/2}\dots T_{1n_1}^{l_{1n_1}/2}} \cdot {T_{02}^{l_{02}/2}\dots T_{0n_0}^{l_{0n_0}/2}} \left( \ii{T_0^{l_0/2}} + {T_1^{l_1/2}} \right).
    $$
    Let us apply Lemma~\ref{LemmaQuasiCoprime} to the last equality. It follows from Proposition~\ref{PropHomCaseGamma} that $\delta(T_{21})$ does not depend on $T_{21}$. Taking into account the coprimeness of the variables, we obtain $\left(\partial T_2^{l_2}/ \partial T_{21}\right)\mid h$.
    Thus, we have obtained derivations of the form~$(c)$ from Construction~\ref{ConstrDer2} when $l_{2c_2}>2$.

    \textbf{Case (C).} In this case, $\delta$ is a replica of one of the derivations $\delta_\infty$ and $\delta_\lambda$, $\lambda\in\KK$, described in Proposition~\ref{PropHomCase222}. The derivations $\delta_\lambda$, where $\lambda\in\{0,\infty,\pm1,\pm\ii\}$, are taken by permutations of variables to derivations proportional to those in this set. Repeating for them the reasoning of case~(B), we obtain derivations of the form~$(c)$ from Construction~\ref{ConstrDer2} when $l_{2c_2}=2$.

    It remains to consider $\delta_\lambda$ when $\lambda\in\KK\setminus \{\,0,\,\pm1,\,\pm\ii\,\}$. Let $\delta=h\delta_\lambda$ for some $h\in\ker\delta_\lambda$. We will show that $d_i=2$ for all $i$ and $\lambda\in\Bbbk$. We have
    \begin{equation}\label{Eq5}
        \delta_\lambda(x)= 2\lambda y+\ii (1+\lambda^2)z,\qquad \delta_\lambda(y) = -2 \lambda x+(1-\lambda^2)z.
        \tag{6.4}
    \end{equation}
    We use once again the fact that $\delta(T_0^{l_0})/\delta(T_1^{l_1})$ lies in the field $(\Quot R)^{\TT}$, every element of which by Lemma~\ref{LemmaRationalInvariants} is representable as a rational function in $
    T_i^{l_i/d_{i2}}/ T_2^{l_2/d_{i2}}$, where $d_{i2}=\gcd(d_i, d_2)$. It is clear that $d_{i}\in\{1,\,2\}$ for all~$i$. On the other hand, if $\delta$ is a replica of $\delta_\lambda$, then
    $$
    {\delta(T_0^{l_0})\over\delta(T_1^{l_1})} = {x\delta_\lambda(x)\over y\delta_\lambda(y)} = {2\lambda + \ii(1+\lambda^2){z\over y}\over-2\lambda + (1-\lambda^2){z\over x}}.
    $$
    Consequently, since $\lambda$ does not depend on $T_{i1}$, we obtain $d_{0}=d_{1}=d_2=2$ and $\lambda\in\Bbbk$.

From relations~(\ref{Eq5}), it follows that
\begin{equation}\label{Eq6}
    \begin{split}
    \delta(T_{01})\cdot{T_{02}^{l_{02}/2}\dots T_{0n_0}^{l_{0n_0}/2}} &=h \left(2\lambda y + \ii (1+\lambda^2)z \right),\\
        \delta(T_{11})\cdot{T_{12}^{l_{12}/2}\dots T_{1n_1}^{l_{1n_1}/2}} &= h\left(-2\lambda x + (1-\lambda^2)z \right),\\
         \delta(T_{21})\cdot{T_{22}^{l_{22}/2}\dots T_{2n_2}^{l_{2n_2}/2}} &=-h(\ii(1+\lambda^2)x+(1-\lambda^2)y).
    \end{split}
    \tag{6.5}
\end{equation}
    Recall that $h \in\KK[(1-\lambda^2)x-\ii(1+\lambda^2)y + 2\lambda z]$ and $\KK$ does not contain the variables $T_{i1}$. Hence any of the relations~(\ref{Eq6}) implies $h\in R$.
    Moreover, by Lemma~\ref{LemmaQuasiCoprime} (applied after a suitable renumbering of the monomials), $h$ is divisible by $\prod_{i=0}^{2}{T_{i2}^{l_{i2}/2}\dots T_{in_i}^{l_{in_i}/2}}$. At the same time, $\delta(T_{i1})$ is divisible by the expression in parentheses on the right-hand side of the corresponding equality in~(\ref{Eq6}). Thus $\delta$ is a replica of the derivation~$(d)$ from Construction~\ref{ConstrDer2}, and $\ker\delta=\ker\delta_\lambda\cap R$.
\end{proof}

\begin{proof}[Proof of Theorem~\ref{TheoremType2Classes}]
    Let $\delta$ be a nonzero ${\TT}$-homogeneous LND on $R$. By Lemma~\ref{LemmaHLND}, either $\delta$ is a replica of ${\partial/\partial S_k}$, or there exists a tuple $C$ satisfying the stated conditions. Let $\KK$ be the algebraic closure of the field obtained from $\Bbbk$ by adjoining all variables $S_k$ and $T_{ij}$ with $j\neq c_i$. Consider the trinomial surface $Z$ over the field $\KK$, given by the same equations as $\Spec R$. By the rigidity criterion~\cite[Theorem~4]{EGSh}, there exist three indices, say $0,1,2$, such that $l_{ic_i}=1$ for $i>2$. Consequently, for $i>2$, the variables $T_{ic_i}$ are expressed linearly in $\KK[Z]$ in terms of $T_{0c_0}^{l_{0c_0}}, T_{1c_1}^{l_{1c_1}}, T_{2c_2}^{l_{2c_2}}$. Thus we may assume that $Z$ is given by the equation \begin{equation}\label{Eq100}
        aT_0^{l_0}+bT_1^{l_1} + c T_2^{l_2}=0,\qquad\qquad a,b,c\in\Bbbk,
        \tag{6.6}
    \end{equation}
    in three-dimensional affine space over $\KK$.
    An explicit formula of $\delta$ is now obtained by repeating the proof of Theorem~\ref{TheoremType2}, taking the coefficients of equation~(\ref{Eq100}) into account. In doing so, the images of the variables $T_{ic_i}$ with $i=0,1,2$ should be multiplied by $\prod_{i>2} {\partial T_i^{l_i}/\partial T_{ic_i}}$, while the images of the variables $T_{ic_i}$ with $i>2$ are determined by the relations between the monomials $T_0^{l_0},\,T_1^{l_1},\,T_i^{l_i}$.
\end{proof}

\section*{Acknowledgements} We sincerely thank our research advisor Sergey Gaifullin for setting the problem, useful discussions, and continuous support throughout the writing of this manuscript.

\end{document}